\documentclass{article}
\usepackage{amsmath}
\usepackage{amssymb, amsthm}
\usepackage{geometry}
\usepackage{hyperref}
\newtheorem{theorem}{Theorem}

\title{The Moore Bound for Regular Simplicial Complexes}
\author{
    Sukrit Chakraborty\thanks{Department of Mathematics, Achhruram memorial College, Jhalda, Purulia, West Bengal, India, \texttt{sukrit049@gmail.com}
        }
}
\date{}

\begin{document}

\maketitle

\begin{abstract}
We derive Moore-type upper bounds for regular simplicial complexes and present logarithmic lower bounds on their diameter based on minimum degree. 
\end{abstract}

In a \(d\)-dimensional simplicial complex \(X\) over the vertex set \(\chi = \{1, 2, \dots, n\}\), where \(d+1 \leq n\), we define \(X_k\) as the set of all \(k\)-simplices on \(\chi\), and the complex \(X\) is formed as \(X := \left( \bigcup_{k=0}^{d-1} X_k \right) \cup X^d\), with \(X^d \subseteq X_d\), ensuring a complete \((d-1)\)-dimensional skeleton. The \textbf{distance} \(d_X(\sigma_1, \sigma_2)\) between two \((d-1)\)-simplices \(\sigma_1\) and \(\sigma_2\) is the length of the shortest path connecting them within the \(d\)-dimensional structure. The \textbf{eccentricity} \(r(\sigma)\) of a \((d-1)\)-simplex \(\sigma\) is the maximum distance from \(\sigma\) to any other \((d-1)\)-simplex in the complex. The \textbf{diameter} of \(X\), denoted \(D=\text{diam}(X)\), is the maximum eccentricity among all \((d-1)\)-simplices in \(X\). More details can be found in \cite{chak25}.

\begin{theorem}[Moore's Bound for simplicial complexes]
Let $r \geq 2$, $N=\binom{n}{d}$, and \( X \) be a connected, undirected simplicial complex with complete $(d-1)$ skeleton that is \( r \)-regular (i.e., every $(d-1)$-simplex has degree \( r \)) and has diameter \( D \). Then the number of $(d-1)$-simplices \( N \) satisfies the inequality:
\[
 N \le 1 + r \sum_{i=0}^{D - 1} (r - 1)^i
= 1 + r \cdot \frac{(r - 1)^D - 1}{r - 2}, \quad \text{for } r \geq 2.
\]
Consequently, the lower bound on the diameter is
\[
D \geq \frac{
\log \left( 1 + \frac{(N  - 1)((r - 1)d - 1)}{rd} \right)
}{
\log \big( (r - 1)d \big)
}.
\]
\end{theorem}
\begin{proof}
Let \( X \) be a \( r \)-regular, connected graph with diameter \( D \). Choose an arbitrary $(d-1)$-simplex \( \sigma_0 \in X_{d-1} \). We will count the maximum number of $(d-1)$-simplices that can be reached from \( \sigma_0 \) within \( D \) steps (facets), assuming that the graph branches out as widely as possible, i.e., with no repeated $(d-1)$-simplices or cycles. Observe that only the $(d-1)$-simplex \( \sigma_0 \) itself is reachable. So, one $(d-1)$-simplex. From \( \sigma_0 \), we can reach at most \( rd \) neighbors. Each of the \( rd \) neighbors can reach at most \( (r - 1)d \) new $(d-1)$-simplices (excluding all $(d-1)$-simplices caintained in \( \tau_0 \supset \sigma_0 \)). So we get at most \( rd(r - 1)d \) new $(d-1)$-simplices. Next, each of the previous level's $(d-1)$-simplices can again reach at most \( (r - 1)d \) new $(d-1)$-simplices. So the number of new $(d-1)$-simplices is at most \( rd[(r - 1)d]^2 \). Continuing this process, at distance \( i \) (for \( 1 \le i \le D \)), the number of new $(d-1)$-simplices is at most \( rd(r - 1)^{i - 1}d^{i - 1} \).

Thus, summing over all distances from 0 to \( D \), the total number of distinct $(d-1)$-simplices that can be reached is at most:

\[
n \le 1 + rd + rd(r - 1)d + rd[(r - 1)d]^2 + \cdots + rd(r - 1)^{D - 1}d^{D - 1}
\]

This is a geometric series with the first term \( 1 \) and ratio \( (r - 1)d \), so:

\[
n \le 1 + rd \sum_{i=0}^{D - 1} [(r - 1)d]^i = 1 + rd \cdot \frac{[(r - 1)d]^D - 1}{(r-1)d - 1}, \quad \text{for } r \geq 2.
\]

Given the inequality
\[
n \leq 1 + rd \cdot \frac{[(r - 1)d]^D - 1}{(r - 1)d - 1}, \quad \text{for } r \geq 2,
\]

Now, subtract 1 from both sides and divide both sides by \( rd \), we get
\[
\frac{n - 1}{rd} \leq \frac{[(r - 1)d]^D - 1}{(r - 1)d - 1}.
\]

Multiply both sides by \( (r - 1)d - 1 \) and adding 1 to both sides gives
\[
\frac{(N  - 1)((r - 1)d - 1)}{rd} + 1 \leq [(r - 1)d]^D.
\]

Taking the logarithm of both sides
\[
\log \left( \frac{N      - 1)((r - 1)d - 1)}{rd} + 1 \right) \leq D \cdot \log \big( (r - 1)d \big).
\]

Hence,
\[
D \geq \frac{
\log \left( 1 + \frac{(N - 1)((r - 1)d - 1)}{rd} \right)
}{
\log \big( (r - 1)d \big)
}.
\]
\end{proof}

\begin{theorem}
Let \(X = (X_{d-1}, X^d)\) be a simplicial complex with minimum degree \(\delta = k \geq 3\) and \(N:=|X_{d-1}| = \binom{n}{d}\). Then, for every $(d-1)$ simplex \(\sigma \in X_{d-1}\), the eccentricity \(r_\sigma(X)\) satisfies
\[
r_\sigma(X) = O\big(\log_{(k-1)d} N\big)
\]
\end{theorem}
\begin{proof}
Starting from vertex \(\sigma\), the number of vertices reachable within distance \(r\) is at least
\[
1 + kd + kd(k-1)d + kd(k-1)^2d^2 + \cdots + kd(k-1)^{r-1}d^{r-1} = 1 + kd \sum_{i=0}^{r-1} (k-1)^id^i.
\]

Evaluating the geometric sum, we get
\[
1 + kd \cdot \frac{(k-1)^rd^r - 1}{(k-1)d - 1}.
\]

Since this quantity is bounded above by the total number of vertices \(N=\binom{n}{d}\), we have
\[
N \geq 1 + kd \cdot \frac{(k-1)^rd^r - 1}{(k-1)d - 1},
\]
which implies
\[
[(k-1)d]^r \leq \frac{[(k-1)d-1][N - 1]}{kd} + 1 = O\left(N\right).
\]

Taking logarithms, we get
\[
r \leq \log_{(k-1)d} (C \cdot N) = O\big(\log_{(k-1)d} N\big),
\]
for some constant \(C\).

Hence, the eccentricity \(r_\sigma(X)\) of any $(d-1)$ simplex \(\sigma\) is at most on the order of \(\log_{(k-1)d} N\).
\end{proof}

\end{document}